\documentclass[11pt]{amsart}
\usepackage{graphicx}
\usepackage{amsaddr}
\usepackage{hyperref}
\usepackage{amsmath,amssymb}
\usepackage{mathrsfs}                       

\large\normalsize

\theoremstyle{plain}
 \newtheorem{theorem}{Theorem}[section]
 
 \newtheorem{lem}{Lemma}[section]
 
\theoremstyle{Definition}
 \newtheorem{exm}{Example}[section]
 
 \newtheorem{dfn}{Definition}[section]
\theoremstyle{remark}

 \numberwithin{equation}{section}

\begin{document}
	\newcommand{\T}{\mathbb{T}}
	\newcommand{\R}{\mathbb{R}}
	\newcommand{\Q}{\mathbb{Q}}
	\newcommand{\N}{\mathbb{N}}
	\newcommand{\Z}{\mathbb{Z}}
	\newcommand{\tx}[1]{\quad\mbox{#1}\quad}

\title [Escaping Set of Hyperbolic Semigroup]{Escaping Set of Hyperbolic Semigroup} 

\author[Bishnu Hari Subedi, Ajaya Singh]{Bishnu Hari Subedi $^1$ \and Ajaya Singh $^2$}
\address{ $^{1, \; 2}$Central Department of Mathematics, Institute of Science and Technology, Tribhuvan University, Kirtipur, Kathmandu, Nepal \\Email: subedi.abs@gmail.com, singh.ajaya1@gmail.com}

\vspace{-.2cm} 
\thanks{\hspace{-.5cm}\tt
	This research work of first author is supported by PhD faculty fellowship from University Grants Commission, Nepal. 
	\hfill }

\maketitle
\thispagestyle{empty}

{\footnotesize
\noindent{\bf Abstract:}  \textit{In this paper,  we mainly study hyperbolic semigroups from which we get non-empty escaping set and Eremenko's conjecture remains valid.  We prove that if each generator of bounded type transcendental semigroup $ S $ is  hyperbolic, then the semigroup is itself hyperbolic and all components of $ I(S) $ are unbounded}. 
\bf AMS (MOS) [2010] Subject Classification.} {37F10, 30D05}

\section{Introduction}
Throughout this paper, we denote the \textit{complex plane} by $\mathbb{C}$ and set of integers greater than zero by $\mathbb{N}$. 
We assume the function $f:\mathbb{C}\rightarrow\mathbb{C}$ is \textit{transcendental entire function} (TEF) unless otherwise stated. 
For any $n\in\mathbb{N}, \;\; f^{n}$ always denotes the nth \textit{iterates} of $f$. Let $ f $ be a TEF. The set of the form
$$
I(f) = \{z\in \mathbb{C}:f^n(z)\rightarrow \infty \textrm{ as } n\rightarrow \infty \}
$$
is called an \textit{escaping set} and any point $ z \in I(S) $ is called \textit{escaping point}. For TEF $f$, the escaping set $I(f)$ was first studied by A. Eremenko \cite{ere}. He himself showed that 
 $I(f)\not= \emptyset$; the boundary of this set is a Julia set $ J(f) $ (that is, $ J(f) =\partial I(f) $);
 $I(f)\cap J(f)\not = \emptyset$; and 
 $\overline{I(f)}$ has no bounded component. By motivating from this last statement, he posed a question: \textit{Is every component of $ I(f) $ unbounded?}. This question is considered  as an important open problem of transcendental dynamics and it is known as \textit{Eremenko's conjecture}. Note that the complement of Julia set $ J(f) $ in complex plane $ \mathbb{C} $ is a \textit{Fatou set} $F(f)$.
 
 Recall that the set $CV(f) = \{w\in \mathbb{C}: w = f(z)\;\ \text{such that}\;\ f^{\prime}(z) = 0\} $  represents the set of \textit{critical values}. The set 
$AV(f)$ consisting of all  $w\in \mathbb{C}$ such that there exists a curve (asymptotic path) $\Gamma:[0, \infty) \to \mathbb{C}$ so that $\Gamma(t)\to\infty$ and $f(\Gamma(t))\to w$ as $t\to\infty$ is called the set of \textit{asymptotic values} of $ f $ and the set
$SV(f) =  \overline{(CV(f)\cup AV(f))}$
is called the \textit{singular values} of $ f $.  
If $SV(f)$ has only finitely many elements, then $f$ is said to be of \textit{finite type}. If $SV(f)$ is a bounded set, then $f$ is said to be of \textit{bounded type}.                                                                                                                             
The sets
$$\mathscr{S} = \{f:  f\;\  \textrm{is of finite type}\} 
\;\;  \text{and}\; \;                                                                                                                                                                                                                                                                                                                                                                                                                                                                                                                                                                                                                                                                                                                                                                                                                                                                                                                                                                                                                                                                                                                                                                                                                                                                                                                                                                                                                                                                                                                                                                                                                                                                                                                                                                                                                                                                                                                                                                                                                                                                                                                                                                                                                                                                                                                                                                                                                                                                                                                                                                                                                                                                                                                                                                                                                                                                                                                                                                                                                                                                                                                     
\mathscr{B} = \{f: f\;\  \textrm{is of bounded type}\}
$$
are respectively called \textit{Speiser class} and \textit{Eremenko-Lyubich class}. 
The \textit{post-singular point} is the point on the orbit of singular value. That is, if $z$ is a singular value of entire function $f$, then $f^n(z)$  is a post-singular point for $n\geq 0$. The set of all post-singular points is called \textit{post-singular set} and it is denoted by 
$$P(f) =\bigcup_{n\geq 0}f^n(SV( f))
$$ 
The entire function $f$ is called \textit{post-singularly bounded} if its post-singular set is bounded and it is called \textit{post-singularly finite} if its post-singular set is finite. A transcendental entire function $ f $ is \textit{hyperbolic} if the post singular set $ P(f) $ is compact subset of the Fatou set $ F(f) $. 

The main concern of this paper is to study of escaping set under transcendental semigroup. So we start our formal study from the notion of transcendental semigroup. Note that for given complex plane $\mathbb{C}$, the set $\text{Hol}(\mathbb{C})$  denotes a set of all holomorphic functions of $ \mathbb{C} $. If $ f\in \text{Hol}(\mathbb{C}) $, then $ f $ is a polynomial or transcendental entire function. The set $\text{Hol}(\mathbb{C})$ forms a semigroup  with semigroup operation being the functional composition. 
 
\begin{dfn}[\textbf{Transcendental semigroup}]
Let $ A = \{f_i: i\in \mathbb{N}\} \subset \text{Hol}(\mathbb{C})$ be a set of transcendental entire functions $ f_{i}: \mathbb{C}\rightarrow \mathbb{C} $. A \textit{transcendental semigroup} $S$ is a semigroup generated by the set $ A $ with semigroup operation being the functional composition. We denote this semigroup by $S = \langle f_{1}, f_{2}, f_{3}, \cdots, f_{n}, \cdots \rangle$. 
\end{dfn}
Here, each $f \in S$ is the transcendental entire function and $S$ is closed under functional composition. Thus $f \in S$ is constructed through the composition of finite number of functions $f_{i_k},\;  (k=1, 2, 3,\ldots, m) $. That is, $f =f_{i_1}\circ f_{i_2}\circ f_{i_3}\circ \cdots\circ f_{i_m}$. 

A semigroup generated by finitely many functions $f_i, (i = 1, 2, 3,\ldots, n) $  is called \textit{finitely generated transcendental semigroup}. We write $S= \langle f_1,f_2,\ldots,f_n\rangle$.
 If $S$ is generated by only one transcendental entire function $f$, then $S$ is \textit{cyclic or trivial  transcendental  semigroup}. We write $S = \langle f\rangle$. In this case each $g \in S$ can be written as $g = f^n$, where $f^n$ is the nth iterates of $f$ with itself. The transcendental semigroup $S$ is \textit{abelian} if  $f_i\circ f_j =f_j\circ f_i$  for all generators $f_{i}$ and $f_{j}$ of $ S $.

Based on the Fatou-Julia-Eremenko theory of a complex analytic function, the Fatou set, Julia set and escaping set in the settings of semigroup are defined as follows.
\begin{dfn}[\textbf{Fatou set, Julia set and escaping set}]\label{2ab} 
The set of normality or the Fatou set of the transcendental semigroup $S$ is defined by
  \[F (S) = \{z \in \mathbb{C}: S\;\ \textrm{is normal in a neighborhood of}\;\ z\}\] 
The \textit{Julia set} of $S$ is defined by $J(S) = \mathbb{C} \setminus F(S)$
 and the \textit{escaping set} of $S$ by 
        \[I(S) = \{z \in \mathbb{C}: \;  f^n(z)\rightarrow \infty \;\ \textrm{as} \;\ n \rightarrow \infty\;\   \textrm{for all}\;\ f \in S\}\]
We call each point of the set $  I(S) $ by \textit{escaping point}.        
\end{dfn} 
It is obvious that $F(S)$ is the largest open subset of  $ \mathbb{C} $  such that semigroup $ S $ is normal. Hence its compliment $J(S)$ is a smallest closed set for any transcendental semigroup $S$. Whereas the escaping set $ I(S) $ is neither an open nor a closed set (if it is non-empty) for any semigroup $S$.
        
If $S = \langle f\rangle$, then $F(S), J(S)$ and $I(S)$ are respectively the Fatou set, Julia set and escaping set in classical iteration theory of complex dynamics. In this situation we simply write: $F(f), J(f)$ and $I(f)$. 
For the existing results of Fatou Julia theory under transcendental semigroup, we refer \cite{kri, kum2, kum1, kum3, poo}. 
\section{Some Fundamental Features of Escaping Set}
 
The following immediate relation between $ I(S) $ and $ I(f) $ for any $ f \in S $ will be clear from the definition of escaping set.
\begin{theorem}\label{1c}
 $I(S) \subset I(f)$ for all $f \in S$  and hence  $I(S)\subset \bigcap_{f\in S}I(f)$.
\end{theorem}
\begin{proof}
Let $ z \in I(S) $, then $ f^{n}(z)\rightarrow \infty  $ as $ n \rightarrow \infty $ for all $ f \in S $. By which we mean   $ z \in I(f) $ for any  $ f \in S $. This immediately follows the second inclusion.
\end{proof}
Note that the above same type of relation (Theorem \ref{1c}) holds between $ F(S) $ and $ F(f) $. However opposite relation holds between the sets $ J(S) $ and $ J(f) $. Poon {\cite[Theorem 4.1, Theorem 4.2] {poo}} proved that the Julia set $ J(S) $ is perfect and $ J(S) = \overline{\bigcup_{f \in S} J(f)} $ for any transcendental semigroup $ S $. From the last relation of above theorem \ref{1c}, we can say that the escaping set may be empty. Note that $I(f)\not = \emptyset$  in classical iteration theory \cite{ere}. Dinesh Kumar and Sanjay Kumar {\cite [Theorem 2.5]{kum2}} have mentioned the following transcendental  semigroup $S$, where $I(S)$ is an empty set.
\begin{theorem}\label{e}
The transcendental entire semigroup $S = \langle f_{1},\;f_{2}\rangle$  generated by two functions $f_{1}$ and $ f_{2} $ from  respectively two parameter families $\{e^{-z+\gamma}+c\;  \text{where}\;  \gamma, c  \in \mathbb{C} \; \text{and}\;  Re(\gamma)<0, \; Re(c)\geq 1\}$ and $\{e^{z+\mu}+d, \; \text{where}\;  \mu, d\in \mathbb{C} \; \text{and}\; Re(\mu)<0,  \; Re(d)\leq -1\}$ of functions  has empty escaping set $I(S)$. \end{theorem}

In the case of non-empty escaping set $ I(S) $, Eremenko's  result \cite{ere},  $\partial I(f) = J(f)$ of classical transcendental dynamics can be generalized to semigroup settings. The following results is due to Dinesh Kumar and Sanjay Kumar {\cite [Lemma 4.2 and Theorem 4.3]{kum2}} which yield the generalized answer in semigroup settings.
\begin{theorem}\label{3}
Let $S$ be a transcendental entire semigroup such that $ I(S) \neq \emptyset $. Then
\begin{enumerate}
\item $int(I(S))\subset F(S)\;\ \text{and}\;\ ext(I(S))\subset F(S) $, where $int$ and $ext$ respectively denote the interior and exterior of $I(S)$.  
 \item $\partial I(S) = J(S)$, where $\partial I(S)$ denotes the boundary of $I(S)$. 
\end{enumerate}
\end{theorem}
This last statement is equivalent to $ J(S)\subset \overline{I(S)} $. 
If $ I(S) \neq \emptyset $, then we {\cite[Theorem 4.6]{sub1}} proved the following result which is a generalization of Eremenko's  result $I(f)\cap J(f) \neq \emptyset $ {\cite[Theorem 2]{ere}} of classical transcendental dynamics to holomorphic semigroup dynamics. 
\begin{theorem}\label{lu1}
Let $S$ be a transcendental semigroup such that $ F(S)$ has a multiply connected component. Then $I(S)\cap J(S) \neq \emptyset $
\end{theorem}

Eremenko and Lyubich \cite{ere1} proved that if transcendental function $ f\in \mathscr{B} $, then $ I(f)\subset J(f) $, and $ J(f) = \overline{I(f)} $. Dinesh Kumar and Sanjay Kumar {\cite [Theorem 4.5]{kum2}} generalized these results to a finitely generated transcendental semigroup of bounded type  as shown below.
\begin{theorem}\label{4}
For every finitely generated transcendental semigroup $ S= \langle f_1,  f_2,  \ldots,f_n\rangle $ in which each generator $f_i $ is of bounded type, then $ I(S)\subset J(S) $ and $ J(S) = \overline{I(S)} $. 
\end{theorem}
\begin{proof}
 Eremenko and Lyubich's  result \cite{ere1} shows that $ I(f) \subset J(f) $ for each $ f\in S $ of bounded type. Poon's result shows  {\cite[Theorem 4.2]{poo}} that $ J(S) = \overline {\bigcup_{f\in S}J(f)}$. Therefore, (from the definition of escaping set and theorem \ref{1c}) for every $ f\in S, \; I(S)\subset I(f)\subset  J(f)\subset J(S)$. 

The next part follows from the facts $ J(S)\subset\overline{I(S)} $ and $ I(S)\subset J(S) $.
\end{proof}

\section{Escaping set of Hyperbolic Semigroup}

The definitions of critical values, asymptotic values and singular values as well as post singularities of transcendental entire functions can be generalized to arbitrary setting of transcendental semigroups.
\begin{dfn}[\textbf{Critical point, critical value, asymptotic value and singular value}]
A point  $z\in \mathbb{C}$ is called \textit{critical point} of $S$ if it is critical point of some $g \in S$. A point $w\in \mathbb{C}$  is called a \textit{critical value} of $S$ if it is a critical value of some $g \in S$. A point $w \in \mathbb{C}$ is called an \textit{asymptotic value} of $S$ if it is an asymptotic value of some $g \in S$. A point  $w\in \mathbb{C}$  is called a \textit{singular value} of $S$ if it is a singular value of some $g \in S$. For a semigroup $ S $, if all $g \in S   $ belongs to $\mathscr{S}$ or $\mathscr{B}$, we call $ S $ a semigroup of class $\mathscr{S}$ or $\mathscr{B}$ (or finite or bounded type).
\end{dfn}
\begin{dfn}[\textbf{Post singularly bounded (or finite) semigroup}]
A transcendental semigroup $ S $ is said to be post-singularly bounded (or post-singularly finite) if each $g \in S$ is post-singularly bounded  (or post-singularly finite). Post singular set of post singularly bounded semi-group $ S $ is the set of the form
$$P(S) =\overline{\bigcup_{f\in S}f^n(SV( f))}
$$ 
\end{dfn}
\begin{dfn}[\textbf{Hyperbolic semigroup}]\label{1m}
An transcendental entire function $f$ is said to be \textit{hyperbolic} if the post-singular set $P(f)$ is a compact subset of $F(f)$. A transcendental  semigroup $S$ is said to be \textit{hyperbolic} if each $g\in S$ is hyperbolic (that is, if $ P(S)$ is a compact subset of  $F(S) $). 
\end{dfn}
 Note that if transcendental semigroup $ S $ is hyperbolic, then each $ f\in S$ is hyperbolic. However, the converse may not true. The fact $ P(f^{k}) = P(f) $ for all $ k \in \mathbb{N} $ shows that $ f^{k} $ is hyperbolic if $ f $ is hyperbolic. The following result has been shown by Dinesh Kumar and Sanjay Kumar {\cite [Theorem 3.16]{kum2}} where Eremenko's conjecture holds.
\begin{theorem}\label{2a}
Let $f \in \mathscr{B}$ periodic with period p and hyperbolic. Let $g =f^n+p, \; n \in  \mathbb{N}$. Then $S =\langle f, g\rangle$ is hyperbolic and all components of $I(S)$ are unbounded.
\end{theorem}

\begin{exm}
$ f(z) = e^{\lambda z} $ is hyperbolic entire function  for each $\lambda \in (0, \frac{1}{e}) $. The semigroup $ S = \langle f, g \rangle  $ where $g =f^m +p$, where $ p = \frac{2 \pi i}{\lambda} $, is hyperbolic transcendental semigroup.
\end{exm}

We have generalized the above theorem \ref{2a} to finitely generated hyperbolic semigroup with some modifications. This theorem will be the good source of non-empty escaping set transcendental semigroup from which the Eremenko's conjecture holds.
 
\begin{theorem}\label{hs1}
Let $ S =\langle f_{1},  f_{2}, \ldots, f_{n} \rangle$ is an abelian bounded type transcendental semigroup in which each $ f_{i} $ is hyperbolic for $ i =1, 2, \ldots, n $. Then semigroup $ S $ is hyperbolic and all components of $ I(S) $ are unbounded.
\end{theorem}

\begin{lem}\label{hs2}
Let $ f $ and $ g $ be transcendental entire functions. Then 
$ SV(f \circ g) \subset SV(f) \cup f(SV(g)) $.
\end{lem}
\begin{proof}
See for instance {\cite[Lemma 2]{ber}}.
\end{proof}
\begin{lem}\label{hs3}
Let $ f $ and $ g $ are permutable transcendental entire functions. Then $ f^{m}(SV(g)) \subset SV(g) $ and $g^{m}(SV(f))\subset SV(f) $ for all $ m \in \mathbb{N} $.
\end{lem}
\begin{proof}
We first prove that $ f(SV(g)) \subset SV(g) $. Then we use induction to prove $ f^{m}(SV(g)) \subset SV(g) $. 
 
Let $ w \in f(SV(g)) $. Then $ w = f(z) $ for some $ z \in SV(g) $. In this case, $ z $ is either critical value or an asymptotic value of function $ g $. 

First suppose that $ z $ is a critical value of $ g $. Then $ z = g(u) $ with $ g^{'}(u) =0 $. Since $ f $ and $ g $  are permutable functions, so $ w = f(z) = f(g(u))= (f\circ g)(u) = (g \circ f)(u) $. Also, $ (f\circ g)^{'}(u) = f^{'}(g(u)) g^{'}(u) =0 $. This shows that $ u $ is a critical point of $ f \circ g =g \circ f $ and $ w $ is a critical value of $ f \circ g =g \circ f $. By permutability of $ f $ and $ g $,  we can write  $f^{'}(g(u)) g^{'}(u) = g^{'}(f(u)) f^{'}(u) =0$ for any critical point $ u $ of $ f \circ g $. Since $ g^{'}(u) =0 $, then either $ f^{'}(u) =0 \Rightarrow  u$ is a critical point of $ f $ or $ g^{'}(f(u)) =0 \Rightarrow f(u) $ is a critical point of $ g$. This shows that $ w = g(f(u)) $ is a critical value of $ g $. Therefore,  $w\in SV(g)$. 

Next, suppose that $ z $ is an asymptotic value of function $ g $. We have to prove that $ w = f(z) $ is also asymptotic value of $ g $. Then there exists a curve $ \gamma: [0, \infty) \to \mathbb{C} $ such that $ \gamma (t) \to \infty $ and $  g(\gamma (t)) \to z $. So,  $  f(g(\gamma (t))) \to f(z) =w $ as $ t \to \infty $ along $ \gamma $. Since $ f\circ g = g \circ f $, so    $ f(g(\gamma (t))) \to f(z) =w\Rightarrow  g(f(\gamma (t))) \to f(z) =w$ as $ t \to \infty $ along $ \gamma $. This shows $ w $ is an asymptotic value of $ g $. This proves our assertion.

Assume that $ f^{k}(SV(g)) \subset SV(g) $ for some $ k \in \mathbb{N} $ with $ k \leq m $. Then 
$$ 
f^{k +1}(SV(g)) =f(f^{k}(SV(g))) \subset f(SV(g)) \subset SV(g) 
$$
Therefore, by induction, for all $ m \in \mathbb{N} $, we must have $ f^{m}(SV(g)) \subset SV(g) $.
The next part $g^{m}(SV(f))\subset SV(f) $ can be proved similarly as above.

\end{proof}
\begin{lem}\label{hs4}
Let $ f $ and $ g $ are two permutable  hyperbolic  transcendental entire functions. Then their composite $ f\circ g $ is also hyperbolic.
\end{lem}
\begin{proof}
We have to prove that $ P(f \circ g) $ is a compact subset of  Fatou set $ F(f \circ g) $. From {\cite[Lemma 3.2]{kum4}}, $ F(f \circ g) \subset F(f) \cap F(g) $. This shows that $ F(f \circ g) $ is a subset of  $ F(f) $ and $ F(g) $. So this lemma will be proved if we prove $ P(f \circ g) $ is a compact subset of $ F(f) \cup F(g) $. By the definition of post singular set of transcendental entire function, we can write
\begin{align*}
P(f \circ g) & =\overline{\bigcup_{m\geq 0}(f \circ g)^m(SV(f \circ g))}\\
                 & = \overline{\bigcup_{m\geq 0} f^{m}(g^{m}(SV(f \circ g)))}\;\;\;\;\;\;\;\;\;\;\;\;\;\; (\text{by using permutabilty of $f$ and $g$}) \\
                 & \subset \overline{\bigcup_{m\geq 0} f^{m}(g^{m}(SV(f) \cup f(SV(g)))} \;\;\;\;\;\;\;\;\;\; (\text{by above lemma \ref{hs2}})\\
                 & = \overline{\bigcup_{m\geq 0} f^{m}(g^{m}(SV(f))) \cup g^{m}(f^{m + 1}(SV(g)))} \\
                 & \subset\overline{\bigcup_{m\geq 0} f^{m}(SV(f)))} \cup \overline{\bigcup_{m\geq 0} g^{m}(SV(g)))}\;\;\;\;(\text{by above lemma \ref{hs3}})\\  
                 & = P(f) \cup P(g)
\end{align*}
Since $ f $ and $ g $ are hyperbolic, so $ P(f) $ and $ P(g) $ are compact subset of $ F(f) $ and $ F(g) $. Therefore, the set $P(f) \cup P(g)  $ must be compact subset of $ F(f) \cup F(g) $.
\end{proof}
\begin{proof}[Proof of the Theorem \ref{hs1}]
Any $ f \in S $ can be written as $ f = f_{i_1}\circ f_{i_2}\circ f_{i_3}\circ \cdots\circ f_{i_m}$. By permutability of each $ f_{i} $, we can rearrange $ f_{i_{j}} $ and ultimately represented by 
$$
f = f_{1}^{t_{1}} \circ f_{2}^{t_{2}} \circ \ldots \circ f_{n}^{t_{n}}
$$
where each $ t_{k}\geq 0 $ is an integer for $ k = 1, 2, \ldots, n $. The lemma \ref{hs4} can be applied repeatably to show each of $f_{1}^{t_{1}}, f_{2}^{t_{2}},\ldots, f_{n}^{t_{n}}  $ is hyperbolic. Again by repeated application of above same lemma, we can say that $f = f_{1}^{t_{1}} \circ f_{2}^{t_{2}} \circ \ldots \circ f_{n}^{t_{n}}$ is itself hyperbolic  and so the semigroup $ S $ is hyperbolic. Next part follows from {\cite[Theorem 3.3]{sub2}} by the assumption of this theorem.
\end{proof}
\textbf{Acknowledgment}: We express our heart full thanks to Prof. Shunshuke Morosawa, Kochi University, Japan for his thorough reading of this paper with valuable suggestions and comments.

\end{document}